\documentclass[a4paper,12pt]{article}
\newcommand{\slim}{\mbox{s-}\lim}
\newcommand{\wlim}{\mbox{w-}\lim}
\usepackage{amsmath,amssymb,mathrsfs}
\usepackage[dvips]{graphicx}
\usepackage{amsthm,bm,%ascmac,%picins,
color}
\makeatletter
\def\figcaption{\def\@captype{figure}\caption}
\makeatother
\makeatletter
 
  \@addtoreset{equation}{section}
 \makeatother

\newtheorem{theorem}{\bf Theorem}[section]
\newtheorem{lemma}{\bf Lemma}[section]
\newtheorem{proposition}{\bf Proposition}[section]
\newtheorem{corollary}[theorem]{\bf Corollary}

\newtheorem{example}{\bf Example}[section]

\title{Asymptotic velocity of \\
a position-dependent quantum walk}
\author{
Akito Suzuki
\thanks{Division of Mathematics and Physics, 
Faculty of Engineering, Shinshu University, Wakasato, Nagano 380-8553, Japan, 
e-mail: akito@shinshu-u.ac.jp
	}
}
\begin{document}

\maketitle

%%%%%%%%%%%%%% abstract %%%%%%%%%%%%%%%%%%%%%%%%%%%%%%%%%%%%%%%%%%%%%%%%%%%%%%%%%%%%%%%%%%%%%%%
\begin{abstract}
We consider 
a position-dependent coined quantum walk on $\mathbb{Z}$
and assume that the coin operator $C(x)$ satisfies
\[ \|C(x) - C_0 \| \leq c_1|x|^{-1-\epsilon},
	\quad x \in \mathbb{Z}\setminus \{0\} \]
with positive $c_1$ and $\epsilon$ and $C_0 \in U(2)$. 
We show that
the Heisenberg operator $\hat x(t)$ of the position operator 
converges to the asymptotic velocity operator $\hat v_+$
so that
\[ \slim_{t \to \infty} {\rm exp}\left( i \xi \frac{\hat x(t)}{t} \right)
	= \Pi_{\rm p}(U) + {\rm exp}(i \xi \hat v_+) \Pi_{\rm ac}(U) \]
provided that $U$ has no singular continuous spectrum. 
Here $\Pi_{\rm p}(U)$ (resp. $\Pi_{\rm ac}(U)$) is the orthogonal projection onto
the direct sum of all 
eigenspaces (resp. the subspace of absolute continuity) of $U$.
We also prove that 
for the random variable $X_t$ denoting the position of a quantum walker
at time $t \in \mathbb{N}$,
$X_t/t$ converges in law to a random variable $V$
with the probability distribution
\[ \mu_V = \|\Pi_{\rm p}(U)\Psi_0\|^2\delta_0
	+ \|E_{\hat v_+}(\cdot) \Pi_{\rm ac}(U)\Psi_0\|^2, \]
where $\Psi_0$ is the initial state, 
$\delta_0$ the Dirac measure at zero, and $E_{\hat v_+}$ the spectral measure of $\hat v_+$.

\end{abstract}
%%%%%%%%%%%%%%%%%%%%%%%%%%%%
%\tableofcontents
%%%%%%%%%%%%%% body %%%%INTRODUCTION%%%%%%%%%%%%%%%%%%%%%%%%%%%%%%%%%%%%%%%%%%%%%%%%%%%%%%%%%%%%%%%%%%%%%%%
\section{Introduction}
%%%%%%%%%%%%%%%%%%%%%%%%%%%%%%%%%%%%%%%
The weak limit theorems for discrete time quantum walks
have been studied in various models (for reviews, see \cite{Ko08, VA15}).
In his papers \cite{Ko02, Ko05},
Konno first proved the weak limit theorem 
for a position-independent quantum walk on $\mathbb{Z}$.
%Using the Fourier analysis,
Grimmett et al \cite{GrJaSc04} simplified the proof 
and extended the result to higher dimensions.
For positon-dependent qunatum walks on $\mathbb{Z}$,
the weak limit theorems were obtained by
Konno et al \cite{KoLuSe13}, 
Endo and Konno \cite{EnKo14}, 
and Endo et al \cite{Enetal14}.

We consider a position-dependent quantum walk on $\mathbb{Z}$
given by a unitary evolution operator $U$:
\[ (U\Psi)(x) = P(x+1)\Psi(x+1) + Q(x-1)\Psi(x-1), \quad x \in \mathbb{Z}, \]
where $\Psi$ is a state vector in the Hilbert space
$\mathcal{H} = \ell^2(\mathbb{Z};\mathbb{C}^2)$
of states and
\[ P(x) = \begin{pmatrix} a(x) & b(x) \\ 0 & 0 \end{pmatrix},
	\quad Q(x) = \begin{pmatrix} 0 & 0 \\ c(x) & d(x) \end{pmatrix} . \]
Let $C(x) = P(x) + Q(x) \in U(2)$ and $S$ be a shift operator
such that $U = SC$.
Suppose that
there exists  a unitary matrix $C_0 = P_0 + Q_0 \in U(2)$ such that
\begin{equation}
\label{Eq_decay}
\|C(x) - C_0\| \leq c_1 |x|^{-1-\epsilon},
	\quad x \in \mathbb{Z} \setminus \{0\} 
\end{equation}
with positive $c_1$ and $\epsilon$ independent of $x$.
Here $\|M\|$ stands for the operator norm of a matrix $M \in M_2(\mathbb{C})$.
A typical example is the quantum walks with one defect 
\cite{Ca12, Ko10, KoLuSe13,W013},
which clearly satisfies \eqref{Eq_decay}.
We note that 
the condition \eqref{Eq_decay} allows not only finite 
but also infinite defects,
whereas the models introduced in \cite{EnKo14, Enetal14}
do not satisfy \eqref{Eq_decay}.
The unitary operator $U_0 = SC_0$ also defines an evolution 
of a position-independent quantum walk on $\mathbb{Z}$
and satisfies 
\[ (U_0\Psi)(x) = P_0 \Psi(x+1) + Q_0\Psi(x-1), \quad x \in \mathbb{Z} \]
with $C_0 = P_0 + Q_0$.
Let $\hat x$ be the position operator defined by
$(\hat x \Psi)(x) = x \Psi(x), \quad x \in \mathbb{Z}$. 
and $\hat x_0(t) = U_0^{-t} \hat x U_0^t$ 
the Heisenberg operator of $\hat x$ at time $t \in \mathbb{N}$
with the evolution $U_0$.
In \cite{GrJaSc04},
Grimmett et al essentially proved that the operator $\hat x_0(t)/t$ 
weakly converges  to the asymptotic velocity operator $\hat v_0$
so that
%%%%%%%%%%%%%%%%
\begin{equation}
\label{w_limit_x0} 
\wlim_{t \to \infty} {\rm exp} \left( i \xi \frac{\hat x_0(t)}{t} \right)
	= {\rm exp} \left({i \xi \hat v_0} \right),
	\quad \xi \in \mathbb{R}. 
\end{equation}
%%%%%%%%%%%%%%%%%%
Let $X^{(0)}_t$ be the random variable denoting the position of a quantum walker
at time $t \in \mathbb{N}$ with the evolution operator $U_0$. 
Then, the characteristic function of $X^{(0)}_t/t$ is given by
\[ \mathbb{E}(e^{i \xi X^{(0)}_t/t}) = \langle \Psi_0, e^{i\xi \hat x_0(t)/t} \Psi_0 \rangle, 
	\quad \xi \in \mathbb{R}, \]
where $\Psi_0$ is the initial state of the quantum walker. 
Hence, \eqref{w_limit_x0} means that
the random variable  $X^{(0)}_t/t$ converges in law to 
a random variable $V_0$,
which represents the linear spreading of the quantum walk: $X^{(0)}_t \sim t V_0$.
%the characteristic function of $X_t/t$ converges to 
%the characteristic function of a random variable $V_0$:
%\[ \lim_{t \to \infty} \mathbb{E}(e^{i \xi X_t/t}) 
%	= \mathbb{E}(e^{i \xi V_0}),
%		\quad \xi \in \mathbb{R}.  \]

In this paper, we derive the asymptotic velocity $\hat v_+$
for the Heisenberg operator $\hat x(t) = U^{-t} \hat x U^t$
with the evolution $U$ of the position-dependent quantum walk.
The decaying condition \eqref{Eq_decay}
implies that $U-U_0$ is a trace class operator
and allows us to prove the existence and completeness of 
the wave operator
\[ W_+ =  \slim_{t \to \infty} U^{-t}U_0^t\Pi_{\rm ac}(U_0) \]
using a discrete analogue of the Kato--Rosenblum Theorem
(See \cite{Si} for details),
where $\Pi_{\rm ac}(U_0)$ is the orthogonal projection onto 
the subspace of absolute continuity of $U_0$.
We also prove that
\begin{equation*}
%\label{s_limit_x0} 
\slim_{t \to \infty} {\rm exp} \left( i \xi \frac{\hat x_0(t)}{t} \right)
	= {\rm exp} \left({i \xi \hat v_0} \right),
	\quad \xi \in \mathbb{R} 
\end{equation*}
under a reasonable condition,
which is essentially the same as that of \cite{GrJaSc04}.
Furthermore, we assume that
$U$ has no singular continuous spectrum.
Then, we prove that
\begin{equation}
\label{eq_20151023} 
\slim_{t \to \infty} {\rm exp}\left( i \xi \frac{\hat x(t)}{t} \right)
	= \Pi_{\rm p}(U) + {\rm exp}(i \xi \hat v_+) \Pi_{\rm ac}(U), 
\end{equation}
where $\Pi_{\rm p}(U)$ is the orthogonal projection onto
the direct sum of all eigenspaces of $U$ and $\hat v_+ = W_+ \hat v_0 W_+^*$.
We believe that the absence of a singular continuous spectrum
can be checked with a concrete example
such as the one-defect model.
As a consequence of \eqref{eq_20151023},
we have the following weak limit theorem.
Let $X_t$ be the random variable denoting the position of a quantum walker
at time $t \in \mathbb{N}$
with the evolution operator $U$ and the initial state $\Psi_0$.
We prove that $X_t/t$ converges in law to a random variable $V$
with a probability distribution
\[ \mu_V = \|\Pi_{\rm p}(U)\Psi_0\|^2\delta_0
	+ \|E_{\hat v_+}(\cdot) \Pi_{\rm ac}(U)\Psi_0\|^2, \]
where $\delta_0$ is the Dirac measure at zero and 
$E_{\hat v_+}$ the spectral measure of $\hat v_+$.

The remainder of this paper is organized as follows.
In Section 2, we present the precise definition of the model
and our results.
Section 3 is devoted to the proof of the existence and completeness
of the wave operator.
In Section 4, we construct the asymptotic velocity.

%%%%%%%%%%%%%%%%%%%%%%%%%%%%%%%%%%%%%
\section{Definition of the model}
%%%%%%%%%%%%%%%%%%%%%%%%%%%%%%%%%%%%
Let $\mathcal{H} = \ell^2(\mathbb{Z};\mathbb{C}^2)$ be the Hilbert space of 
the square-summable functions $\Psi:\mathbb{Z} \to \mathbb{C}^2$.
%with the norm
%\[ \|\Psi\|_\mathcal{H}
%	= \left(\sum_{x \in \mathbb{Z}} \|\Psi(x)\|_{\mathbb{C}^2}^2 \right)^{1/2} \]
We define a shift operator $S$
and a coin operator $C$ on $\mathcal{H}$ as follows.
For a vector
$\Psi = \begin{pmatrix} \Psi^{(0)} \\ \Psi^{(1)} \end{pmatrix} \in \mathcal{H}$,
$S\Psi$ is given by
\[ (S\Psi)(x) = \begin{pmatrix} \Psi^{(0)}(x+1) \\ \Psi^{(1)}(x-1) \end{pmatrix},
	\quad x \in \mathbb{Z}. \]
Let $\{C(x)\}_{x \in \mathbb{Z}} \subset U(2)$ be a family of unitary matrices
with
\[ C(x) = \begin{pmatrix} a(x) & b(x) \\ c(x) & d(x) \end{pmatrix}. \]
$C\Psi$ is given by
\[ (C \Psi)(x) = C(x) \Psi(x), \quad x \in \mathbb{Z}. \]
We define an evolution operator as $U = SC$.
$U$ satisfies
\[ (U\Psi)(x) = P(x+1)\Psi(x+1) + Q(x-1)\Psi(x-1), \quad x \in \mathbb{Z} \]
with
\[ P(x) = \begin{pmatrix} a(x) & b(x) \\ 0 & 0 \end{pmatrix},
	\quad Q(x) = \begin{pmatrix} 0 & 0 \\ c(x) & d(x) \end{pmatrix}. \]

For a matrix $M \in M(2,\mathbb{C})$,
we use $\|M\|$ to denote the operator norm in $\mathbb{C}^2$: 
$\|M\| = \sup_{\|{\bm x}\|_{\mathbb{C}^2}=1} \|M{\bm x}\|_{\mathbb{C}^2}$.
%%%%Assumption%%%%%%%%%%%%%%%%%%%%%%%%%
We suppose that:
\begin{itemize}
\item[(A.1)]
There exists a unitary matrix 
$C_0 = \begin{pmatrix} a_0 & b_0 \\ c_0 & d_0 \end{pmatrix}\in U(2)$
such that
\[ \|C(x) - C_0\| \leq c_1|x|^{-1-\epsilon},
	\quad x \in \mathbb{Z} \setminus \{0\} \]
with some positive $c_1$ and $\epsilon$ independent of $x$.
\end{itemize}
%%%%%%%%%%%%%%%%%%%%%%%%%%%%%%%%%%%%%%%%%%%%%%%%
%Henceforth, we simply write the multiplication operator on $\mathcal{H}$ 
%by a constant matrix $M \in M(2,\mathbb{C})$ as $M$:
%For $\Psi \in \mathcal{H}$, $M\Psi$ is given by
%\[ (M\Psi)(x) = M\Psi(x), \quad x \in \mathbb{Z}. \]
We denote by $\mathscr{T}_1$ the set of trace class operators.
%%%Lemma%%%%%%%%%%%%%%%%%%%%%%%%%%%%
\begin{lemma}
\label{lemT1}
{\rm
Let $U$ satisfy (A.1) and set $U_0 = SC_0$.
Then, $U - U_0 \in \mathscr{T}_1$. 
}
\end{lemma}
%%%%%%%%%%%%%%%%%%%%%%%%%%%%%%%%%%%%
\begin{proof}
Let $T = U-U_0$ and $T(x) = C(x) - C_0$. Then 
\begin{equation}
\label{T*T} 
T^*T = (C-C_0)^*(C-C_0)
\end{equation}
is the multiplication operator by the matrix-valued function $T(x)^*T(x)$.
Let $t_i(x)$ ($i=1,2$) be the eigenvalues 
of the Hermitian matrix $T(x)^*T(x) \in M(2,\mathbb{C})$
and take an orthonormal basis (ONB) $\{ \tau_i(x)\}_{i=1,2}$ of corresponding eigenvectors
for all $x \in \mathbb{Z}$.
We use $|\xi \rangle \langle \eta|$ to denote 
the operator on $\mathcal{H}$ defined by 
$|\xi \rangle \langle \eta|\Psi = \langle \eta, \Psi \rangle \xi$.
Then, we have
\begin{equation} 
T^* T 
= \sum_{i=1,2} \sum_{x \in \mathbb{Z}} 
	t_i(x) |\tau_{i,x} \rangle \langle \tau_{i,x}|,
\end{equation}
where $\{\tau_{i,x}\}$ is the ONB given by 
\[ \tau_{i,x}(y) = \delta_{xy} \tau_i(x), \quad y \in \mathbb{Z}. \]
Since $T^*(x)T(x) \geq 0$, we have $t_i(x) \geq 0$.
By (A.1), we know that
\[ \max_{i=1,2} t_i(x) \leq c_1^2 |x|^{-2-2\epsilon}. \]
Hence, we have
\[ {\rm Tr}|T| 
	= \sum_{x \in \mathbb{Z}} \sum_{i=1,2} t_i(x)^{1/2}
	\leq  2 c_1 \sum_{x \in \mathbb{Z}} |x|^{-1-\epsilon} 
	< \infty, \]
which means that $T \in \mathscr{T}_1$.
Since $\mathscr{T}_1$ is an ideal,
$U-U_0 = ST \in \mathscr{T}_1$. 
\end{proof}
%%%Example%%%%%%%%%%%%%%%%%%%%%%%%%%%%%%%
\begin{example}[one-defect model]
{\rm
Let $C_0, C_0^\prime \in U(2)$ be unitary matrices with $C_0 \not= C_0^\prime$
and set
\[ C(x) = \begin{cases} C_0^\prime, & x = 0 \\ 
			C_0, & x \not=0. \end{cases} \]
$U = SC$ satisfies (A.1),
because $C(x) - C_0 = 0$ if $x \not= 0$.
}
\end{example}
%%%Example%%%%%%%%%%%%%%%%%%%%%%%%%%%%%%%
\begin{example}
{\rm
Let $C_0 \in U(2)$ be a unitary matrix and
$\{C(x)\} \subset U(2)$ a family of unitary matrices.
Assume that
\[ \max_{i,j} |(C(x) - C_0)_{ij}| \leq c_1|x|^{-1-\epsilon}, 
	\quad x \in \mathbb{Z} \setminus \{0\}, \]
where $M_{ij}$ denotes the $ij$-component of a matrix $M$.
Then, $U = SC$ satisfies (A.1),
because all norm on a finite dimensional vector space are equivalent.
}
\end{example}
%%%%%%%%%%%%%%%%%%%%%%%%%%%%%%%%%%%%%%%%%
We prove the following theorem in Section 3
using a discrete analogue of the Kato--Rosenblum theorem.
%%Theorem%%%%%%%%%%%%%%%%%%%%%%%%%%%%%%%%%%%%%%%
\begin{theorem}
\label{thm01}
{\rm
Let $U$ and $U_0$ be as above and assume that (A.1) holds. 
Then
\[ W_+ = \slim_{t \to \infty} U^{-t} U_0^t \Pi_{\rm ac}(U_0) \]
exists and is complete.
}
\end{theorem}
%%%%%%%%%%%%%%%%%%%%%%%%%%%%%%%%%%%%%%%%%%%%

In what follows,
we introduce the asymptotic velocity $\hat v_0$, obtained first in \cite{GrJaSc04},
of the quantum walk with the evolution $U_0$ as follows.
Let 
\[ \hat U_0(k) = \begin{pmatrix} e^{ik} & 0 \\ 0 & e^{-ik} \end{pmatrix} C_0,
	\quad k \in [0,2\pi). \]
Since $\hat U_0(k) \in U(2)$,
$\hat U_0(k)$ is represented as
\[ \hat U_0(k) = \sum_{i=1,2}
	\lambda_i(k) |u_j(k) \rangle \langle u_j(k)|, \]
where $\lambda_j(k)$ is an eigenvalue of $\hat U_0(k)$
and $u_j(k)$ is the corresponding eigenvector
with $\|u_j(k)\|=1$.
The function $k \mapsto e^{ik}$ is analytic,
and so is $\lambda_j(k)$.
We need the following assumption on $u_j(k)$:
%%%Assumption%%%%%%%%%%%%%%%%%%%%%%%%%%%%%%%
\begin{itemize}
\item[(A.2)] The functions $k \mapsto u_j(k)$ are continuously differentiable 
in $k$ with
\[ \sup_{k \in [0,2\pi)} \left\|\frac{d}{dk} u_j(k) \right\|_{\mathbb{C}^2} < \infty. \]
\end{itemize}
%%%%%%%%%%%%%%%%%%%%%%%%%%%%%%%%%%%%%%%%%%%
Let $\mathcal{K}$ be the Hilbert space of 
square integrable functions
$f:[0,2\pi) \to \mathbb{C}^2$ with norm
\[ \|f\|_{\mathcal{K}}
	= \left( \int_0^{2\pi} \frac{dk}{2\pi} \|f(k)\|_{\mathbb{C}^2}^2 \right)^{1/2}. \]
Let $\mathscr{F}:\mathcal{H} \to \mathcal{K}$ be
the discrete Fourier transform given by
\[ (\mathscr{F}\Psi)(k) 
	= \sum_{x \in \mathbb{Z}} e^{-ik\cdot x} \Psi(x),
		\quad \Psi \in \mathcal{H}. \]
We also use 
$\hat \Psi (k) = \begin{pmatrix} \hat \Psi^{(0)}(k) \\  \hat \Psi^{(1)}(k) \end{pmatrix}$
to denote the Fourier transform of $\Psi$.
The asymptotic velocity $\hat v_0$ is the self-adjoint operator defined by
\[ \hat v_0 = \mathscr{F}^{-1}
	\left( \int_{[0,2\pi)}^\oplus  \frac{dk}{2\pi} \sum_{j=1,2}
		\left(\frac{ i \lambda_j^\prime(k) }{\lambda_j(k)} \right)
			|u_j(k) \rangle \langle u_j(k)|
	\right) \mathscr{F} \]
The position operator $\hat x$ is a self-adjoint operator
defined by 
\[ (\hat x \Psi)(x) = x \Psi(x), \quad x \in \mathbb{Z} \]
with domain
\[ D(\hat x) 
	= \left\{ \Psi \in \mathcal{H} ~\Big|~
		\sum_{x \in \mathbb{Z}} |x|^2 \|\Psi(x)\|_{\mathbb{C}^2}^2 < \infty \right\}. \]
Let $\hat x_0(t) = U_0^{-t} \hat x U_0^t$
be the Heisenberg operator of $\hat x$ for the evolution $U_0$.  
%%%%Theorem%%%%%%%%%%%%%%%%%%%%%%%%%%%%%%
\begin{theorem}
\label{thm02}
{\rm
Let $\hat v_0$ and $\hat x_0$ be as above.
Suppose that (A.2) holds.
Then,
\begin{equation}
\label{Eq_exp} 
\slim_{t \to \infty} 
	{\rm exp}\left(i \xi \frac{\hat x_0(t)}{t} \right)
		= {\rm exp}(i \xi \hat v_0 ), 
	\quad \xi \in \mathbb{R}. 
\end{equation}
}
\end{theorem}
%%%%%%%%%%%%%%%%%%%%%%%%%%%%%%%%%%%%%%%%%%%%%%
\begin{proof}%[Proof of Theorem \ref{thm02}]
By \cite[Theorem VIII.21]{RSI},
\eqref{Eq_exp}  holds if and only if 
\begin{equation*} 
%\label{Eq_resol}
\slim_{t \to \infty} \left( \frac{\hat x_0(t)}{t} - z \right)^{-1}
	 = (\hat v_0 - z)^{-1},
	\quad z \in \mathbb{C} \setminus \mathbb{R},
\end{equation*}
which is proved in Subsection 4.1.  
\end{proof}
%%Example%%%%%%%%%%%%%%%%%%%%%%%%%%%%%%%%%%%%%%%%%%
\begin{example}
\label{ex_2.3}
{\rm
\begin{itemize}
\item[(i)] Let $C_0 = \begin{pmatrix} 0 & 1 \\ 1 & 0 \end{pmatrix}$.
Then, $\hat U_0(k)$ has eigenvalues $1$ and $-1$,
which are independent of $k$. 
By definition, $\hat v_0$ = 0.
Hence, the random variable $X^{(0)}_t/t$ converges in law
to a random variable $V_0$ with a probability distribution $\delta_0$.
\item[(ii)] Let $C_0 = \begin{pmatrix} 1 & 0 \\ 0 & -1 \end{pmatrix}$.
$\hat U_0(k)$ has eigenvalues $e^{ik}$ and $-e^{-ik}$.
Hence, $\hat v_0$ has eigenvalues $-1$ and $1$.
The random variable $X^{(0)}_t/t$ converges in law
to a random variable $V_0$ with a probability distribution 
$\|\Psi^{(0)}\|^2 \delta_{-1} + \|\Psi^{(1)}\|^2 \delta_1$. 
\item[(iii)] Let $C_0$ be the Hadamard matrix.
The eigenvalues of $\hat U_0(k)$ are given by 
$\lambda_j(k) = ((-1)^j w(k) + i \sin k)/\sqrt{2}$ ($j=1,2$),
where $w(k) = \sqrt{1+\cos^2 k}$.
Hence, $\hat v_0$ has no eigenvalue.
The corresponding eigenvectors 
\[ u_j(k) 
	= \sqrt{\frac{w(k)+(-1)^j \cos k}{2w(k)}} 
		\begin{pmatrix} e^{ik} \\ (-1)^j w(k) - \cos k \end{pmatrix} \]
form an ONB of $\mathbb{C}^2$ 
and satisfy (A.2).
The random variable $X^{(0)}_t/t$ converges in law
to a random variable $V_0$ with a probability distribution 
$\|E_{\hat v_0}(\cdot)\Psi_0\|^2$,
where $E_{\hat v_0}$ is the spectral measure of $\hat v_0$.
Let us consider the Hadmard walk starting from the origin.
Let the initial state $\Psi_0$ satisfy 
$\Psi_0(0)= \begin{pmatrix} \alpha \\ \beta \end{pmatrix}$
($|\alpha|^2 + |\beta|^2=1$)
and $\Psi(x) = 0$ if $x\not=0$.
Then, 
\[ d \| E_{\hat v_0}(v)\Psi_0 \|^2
	= (1 - c_{\alpha,\beta} v) f_K\left(v;\frac{1}{\sqrt{2}}\right) dv, \]
where 
$c_{\alpha, \beta} 
= |\alpha|^2-|\beta|^2+\alpha \bar{\beta} + \bar{\alpha} \beta$,
\[ f_K(v;r) 
= \frac{\sqrt{1-r^2}}{\pi(1-v^2)\sqrt{r^2-v^2}}I_{(-r,r)}(v)
\]
is the Konno function, and $I_A$ is the indicator function of a set $A$.
For more details, the reader can consult \cite{GrJaSc04, Ko08}.
\end{itemize}
}
\end{example}
%%%%%%%%%%%%%%%%%%%%%%%%%%%%%%%%%%%%%%%%%

Let $\hat x(t) = U^{-t} \hat x U$ be the Heisenberg operator of $\hat x$
and define the asymptotic velocity $\hat v_+$ 
for the evolution $U$ by
\[ \hat v_+ = W_+ \hat v_0 W_+^*. \] 
We need the following assumption:
%%%%%%%%%%%%%%%%%%%%%%%%%%%%%%%%%%%%%%%
\begin{itemize}
\item[(A.3)] The singular continuous spectrum of $U$ is empty.
\end{itemize}
We are now in a psition to state our main result,
which is proved in Subsection 4.2. 
%%%%%%%%%%%%%%%%%%%%%%%%%%%%%%%%%%%%%%
\begin{theorem}
\label{thm03}
{\rm
Let $\hat x(t)$ and $\hat v_+$ be as above.
Suppose that (A.1) - (A.3) hold.
Then,
\[ \slim_{t \to \infty} {\rm exp} \left( i \xi \frac{\hat x(t)}{t} \right)
	= \Pi_{\rm p}(U) 
	+{\rm exp} \left(i \xi \hat v_+ \right) \Pi_{\rm ac}(U),
		\quad \xi \in \mathbb{R}. \]
}
\end{theorem}
%%%%%%%%%%%%%%%%%%%%%%%%%%%%%%%%%%%%%%%%
Let $X_t$ be the random variable denoting the position of the walker
at time $t \in \mathbb{N}$ with the initial state $\Psi_0$.
We use $\Pi_{\rm p}(U)$ to denote the orthogonal projection onto
the direct sum of all eigenspaces of $U$
and $E_{A}$ to denote the spectral projection of 
a self-adjoint operator $A$.
%%%%%%%%%%%%%%%%%%%%%%%%%%%%%%%%%%%%%%%%
\begin{corollary}
\label{coro_thm03}
{\rm
Let $X_t$ be as above.
Suppose that (A.1) - (A.3) hold.
Then,
$X_t/t$ converges in law to a random variable $V$
with a probability distribution 
\[ \mu_V = \| \Pi_{\rm p}(U)\Psi_0\|^2 \delta_0 
	+ \|E_{\hat v_+}(\cdot)\Pi_{\rm ac}(U)\Psi_0\|^2, \]
where $\delta_0$ is the Dirac measure at zero.
}
\end{corollary}
%%%%%%%%%%%%%%%%%%%%%%%%%%%%%%%%%%%%%%%%%%%
\begin{proof}
From Theorem \ref{thm01}, $\slim_{t \to \infty} U_0^{-t} U^t ~\Pi_{\rm ac}(U)$
exists and is equal to $W_+^*$.
Then, $W_+$ is unitary from ${\rm Ran}W_+^* = {\rm Ran}\Pi_{\rm ac}(U_0)$ 
to ${\rm Ran}W_+ = {\rm Ran}\Pi_{\rm ac}(U)$.
Since, by Lemma \ref{lem_com}, $U_0$  is strongly commuting with $\hat v_0$, 
we know, from the intertwining property $UW_+ = W_+ U_0$,
that $U$ is also strongly commuting with $\hat v_+$.
Hence, $\hat v_+$ is strongly commuting with
%$\Pi_{\rm p}(U)$ and 
$\Pi_{\rm ac}(U)$ and
$e^{i \xi \hat v_+} \Pi_{\rm ac}(U) = \Pi_{\rm ac}(U) e^{i \xi \hat v_+}$. 
%Since, by (A.3), $\mathcal{H}$ is decomposed into
%the direct sum
%$\mathcal{H} = {\rm Ran}\Pi_{\rm p}(U) \oplus {\rm Ran}\Pi_{\rm ac}(U)$, 
%$e^{i\xi \hat v_+}$ is decomposed into
%$e^{i\xi \hat v_+} = \Pi_{\rm p}(U) \oplus e^{i\xi v_+} \Pi_{\rm ac}(U)$.
Hence, by Theorem \ref{thm03}, 
${\rm exp}(i\xi \hat x(t)/t) \Psi_0$ converges strongly to
$\Pi_{\rm p}(U)\Psi_0 + e^{i \xi \hat v_+} \Pi_{\rm ac}(U)\Psi_0$
and
\begin{align*}
\lim_{t \to \infty} \mathbb{E}(e^{i \xi X_t/t})
& = \langle \Psi_0, \Pi_{\rm p}(U)\Psi_0 + e^{i \xi \hat v_+} \Pi_{\rm ac}(U)\Psi_0 \rangle \\
& = \|\Pi_{\rm p}(U)\Psi_0\|^2 + 
	\int_{-\infty}^\infty e^{i \xi v} 
		d\| E_{\hat v_+}(v) \Pi_{\rm ac}(U)\Psi_0\|^2 \\
& = \int_{-\infty}^\infty e^{i\xi v} d\mu_V(v),
\end{align*}
which proves the corollary.
\end{proof}
%%%%%%%%%%%%%%%%%%
\begin{example}
{\rm
Let $C_0$ be the Hadmard matrix
and $C(x)$ satisfy (A.1).
As seen in Example \ref{ex_2.3} (iii),
(A.2) is satisfied and
the spectrum of $U_0$ is purely absolutely continuous.
Let $\Psi_+ \in \mathcal{H}$ satisfy 
$\Psi_+(0) 
= \begin{pmatrix} \alpha \\ \beta \end{pmatrix}$
($|\alpha|^2 + |\beta|^2 = 1$)
and $\Psi_+(x) = 0$ if $x\not=0$. 
By Example \ref{ex_2.3}, 
\[ d\|E_{\hat v_+}(v) \Pi_{\rm ac}(U) W_+ \Psi_+\|^2
	= d\|E_{\hat v_0}(v) \Psi_+ \|^2
	=(1- c_{\alpha, \beta} v) f_K\left(v;\frac{1}{\sqrt{2}}\right)dv. \]
Let $\Psi_{\rm p} \in {\rm Ran} \Pi_{\rm p}(U_0)$
be a unit vector
and take the initial state $\Psi_0$ as 
$\Psi_0 = C_1 \Psi_{\rm p} + C_2 W_+ \Psi_+$
($|C_1|^2 + |C_2|^2 = 1$).
Suppose that $U = SC$ satisfies (A.3).
By Corollary \ref{coro_thm03},
$X_t/t$ converges in law to $V$ with a probability distribution
$\mu_V$ and
\[ \mu_V(dv)
	= |C_1|^2 \delta_0(dv) 
		+ |C_2|^2  (1- c_{\alpha, \beta} v) 
			f_K\left(v;\frac{1}{\sqrt{2}}\right)dv. \]
}
\end{example}

%%%Section%%%%%%%%%%%%%%%%%%%%%%%%%%%%%%%%%%%
\section{Wave operator}
%%%%%%%%%%%%%%%%%%%%%%%%%%%%%%%%%%%%%%%%%%%
To prove Theorem \ref{thm01}, 
we use the following general proposition:
%%%%%Proposition%%%%%%%%%%%%%%%%%%%%%%%%%
\begin{proposition}
\label{propWave}
{\rm
Let $U$ and $U_0$ be unitary operators 
on a Hilbert space $\mathcal{H}$
and suppose that $U-U_0 \in \mathscr{T}_1$.
The following limit exists: 
\[ W_+ = \slim_{t \to \infty} U^{-t} U_0^t \Pi_{\rm ac}(U_0) \]
}
\end{proposition}
%%%%%%%%%%%%%%%%%%%%%%%%%%%%%%%%%%%
\begin{proof}[Proof of Theorem \ref{thm01}]
Since, by Lemma \ref{lemT1}, $U-U_0 \in \mathscr{T}_1$,
the wave operator $W_+$ exists.
If we interchange the roles of $U$ and $U_0$,
then the proposition says that
the limit $\slim_{t \to \infty} U_0^{-t} U^t \Pi_{\rm ac}(U)$ also exists,
which implies that $W_+$ is complete.
This completes the proof.
\end{proof}
%%%%%%%%%%%%%%%%%%%%%%%%%%%%%%%%%%%%%%

In the remainder of this section,
we suppose that $U-U_0 \in \mathscr{T}_1$ 
and prove Proposition \ref{propWave}. 
This is done by a discrete analogue of \cite[Theorem 6.2.]{Si}.
We use $\mathcal{H}_{\rm ac}$
and $\mathcal{H}_{\rm p}$
to denote the subspaces of absolute continuity and 
the direct sum of all eigenspaces of $U_0$. 
Let $E_0$ be the spectral measure of $U_0$
with $E_0([0,2\pi)) = I$.
Let
\[ \mathcal{H}_{\rm ac,0}
	= \{ \psi \in \mathcal{H}_{\rm ac} \mid
		d \|E_0(\lambda) \psi\|^2 = G_\psi(\lambda)^2 d\lambda
		\mbox{ and } G_\psi \in L^2 \cap L^\infty \}, \]
where $L^2 = L^2([0,2\pi))$ and  $L^\infty = L^\infty([0,2\pi))$.
Although the following lemma may be well known,
we give proofs for completeness.
%%%%Lemma%%%%%%%%%%%%%%%%%
\begin{lemma}
\label{lemdense}
{\rm
$\mathcal{H}_{\rm ac,0}$ is dense in $\mathcal{H}_{\rm ac}$.
}
\end{lemma}
%%%%%%%%%%%%%%%%%%%%%%%%%%%
\begin{proof}
For all $\psi \in \mathcal{H}_{\rm ac}$,
there exists a positive function $F \in L^1$ such that
$d\|E_0(\lambda)\psi\|^2 = F (\lambda) d\lambda$. 
Let $B_n = F^{-1}([0,n])$, and let $\chi_{B_n}$ be the characteristic function of $B_n$.
We set $G_n = \sqrt{F} \chi_{B_n}$ and $\psi_n = E_0(B_n)\Psi$.
Then $G_n \in L^2 \cap L^\infty$
and $\|E_0(B) \psi_n\|^2 = \int_B G_n(\lambda)^2 d\lambda$. 
Hence, $\psi_n \in \mathscr{H}_{\rm ac,0}$
and $\psi = \lim_n \psi_n$. 
This completes the proof.
\end{proof}
%%%Lemma%%%%%%%%%%%%%%%%%%%%%%%%%
\begin{lemma}
\label{lemac0}
{\rm
Let $\phi \in \mathcal{H}$ and $\psi \in \mathcal{H}_{\rm ac,0}$.
Then,
\[ \sum_{t \in \mathbb{Z}}
	|\langle \phi, U_0^t \psi \rangle|^2
		\leq 2\pi \|\phi\|^2 \sup_{\lambda} G_\psi(\lambda)^2. \] 
}
\end{lemma}
%%%%%%%%%%%%%%%%%%%%%%%%%%%%%%%
\begin{proof}
Let $\psi \in \mathcal{H}_{\rm ac,0}$
and $\mathcal{L} = L^2([0,2\pi), G^2_\psi(\lambda) d\lambda)$.
Let $H_0$ be the self-adjoint operator defined by
$\langle \xi, H_0 \eta \rangle 
= \int_0^{2\pi} \lambda d\langle \xi, E_0(\lambda) \eta \rangle$
($\xi, \eta \in \mathcal{H}$).
Let $\mathscr{U}:\mathcal{L} \to \mathcal{H}$ be an injection
defined by $\mathscr{U}f = f(H_0)\psi$ ($f \in \mathcal{L}$).
Then $\mathscr{U} 1 = \psi$ and $\mathscr{U} e^{it\lambda} = U_0^t\psi$
($t \in \mathbb{N}$).
We use $\Pi$ to denote the orthogonal projection onto $U\mathcal{L}$.
Let $\phi \in \mathcal{H}$ and $F = \mathscr{U}^{-1}\Pi\phi \in \mathcal{L}$.
Then we have
\begin{align*} 
\langle \phi, U_0^t\psi \rangle
	= \int_0^{2\pi}  e^{i t\lambda} \bar F(\lambda) G_\psi(\lambda)^2 d\lambda 
	= 2\pi \widehat{\bar F G_\psi^2}(t).
\end{align*}
Hence, by Parseval's identity, we obtain 
\begin{align*} 
& \sum_{t \in \mathbb{Z}} |\langle \phi, U_0^t \psi \rangle|^2
	%& = (2\pi)^2 \sum_{t \in \mathbb{Z}} |\widehat{\bar F G_\psi^2}(t)|^2 \\
	 = 2\pi \int_0^{2\pi}  |\bar F(\lambda) G_\psi(\lambda)^2|^2 d\lambda \\
	&\qquad \leq  2\pi \sup_{\lambda} G_\psi(\lambda)^2
		\int_0^{2\pi}  |\bar F(\lambda)|^2 G_\psi(\lambda)^2 d\lambda 
	 \leq  2\pi \sup_{\lambda} G_\psi(\lambda)^2 \|\Pi \phi\|^2.
\end{align*}
This completes the proof.
\end{proof}
%%%Lemma%%%%%%%%%%%%%%%%%%%%%%%%%%%%%%%%%%%%%%%%%%%%%%%%%%%
Let $W_t = U^{-t} U_0^t $.
\begin{lemma}
\label{lemslimonac}
{\rm
Let $t, s \in \mathbb{N}$ ($s\not=t$). Then,
$\slim_{r \to \infty} (W_t - W_s)U_0^r\Pi_{\rm ac}(U_0) = 0$.
}
\end{lemma}
%%%%%%%%%%%%%%%%%%%%%%%%%%%%%%%%%%%%%%%%%%%%%%
\begin{proof}
For $t, s \in \mathbb{N}$ ($t>s$), we have
$W_t = \sum_{k=s+1}^t (W_k - W_{k-1}) + W_s$
and $W_k - W_{k-1} = U^{-k}(-T)U_0^{k-1}$,
where $T = U-U_0 \in \mathscr{T}_1$.
Since $\mathscr{T}_1$ is an ideal, we know that
\[ W_t -W_s = \sum_{k=s+1}^t U^{-k} (-T) U_0^{k-1}
	\in \mathscr{T}_1. \]
In particular, $W_t-W_s$ is compact.
Let $H_0$ be the self-adjoint operator defined in 
the proof of Lemma \ref{lemac0}.
Since $\wlim_{r \to \infty} e^{irH_0}\Pi_{\rm ac}(H_0) = 0$,
we have
\[ \slim_{r \to \infty} (W_t-W_s) U_0^r \Pi_{\rm ac}(U_0)
	= \slim_{r \to \infty} (W_t-W_s) e^{irH_0}\Pi_{\rm ac}(H_0) = 0. \]
This completes the proof.
\end{proof}	
%%%%%Proof%%%%%%%%%%%%%%%%%%%%%%%%%%%%%%%%%%%%%		
\begin{proof}[Proof of Proposition \ref{propWave}]
By Lemma \ref{lemdense},
it suffices to prove that, for $\psi \in \mathcal{H}_{\rm ac,0}$,
\[ \|(W_t - W_s)\psi\| \to 0, 
	\quad t,s \to \infty . \]
Because
\[ \|(W_t - W_s)\psi\|^2 
	= \langle \psi, W_t^*(W_t - W_s)\psi \rangle
		- \langle  \psi, W_s^*(W_t - W_s)\psi \rangle, \]
we need only to prove that
\[ \langle \psi, W_t^*(W_t - W_s)\psi \rangle \to 0, 
	\quad t,s \to \infty . \]
By direct calculation, we have, for $r > 1$,
\begin{align*}
& W_t^*(W_t - W_s) - U_0^{-r} W_t^* (W_t - W_s) U_0^r \\
& =U_0^{-r} W_t^* W_s U_0^r - W_t^* W_s  \\
& = \sum_{k=0}^{r-1}
	\left( U_0^{-k-1} W_t^* W_s U_0^{k+1} - U_0^{-k} W_t^* W_sU_0^k \right).
\end{align*}
Since
\[ U_0^{-k-1} W_t^* W_s U_0^{k+1} - U_0^{-k} W_t^* W_sU_0^k 
	= U_0^{-k-t-1}
		\left( TU^{t-s} - U^{t-s} T\right) U_0^{s+k}, \]
we obtain
\begin{align*}  
& W_t^*(W_t - W_s) - U_0^{-r} W_t^* (W_t - W_s) U_0^r \\
& = \sum_{k=0}^{r-1}
	U_0^{-k-t-1}
		\left( TU^{t-s} - U^{t-s} T\right) U_0^{s+k}.
\end{align*}
Since, by Lemma \ref{lemslimonac},
$\slim_{r \to \infty}U_0^{-r} W_t^* (W_t - W_s) U_0^r \psi= 0$,
we have
\begin{align*}
& W_t^*(W_t - W_s)\psi
	= \sum_{k=0}^{\infty}
	U_0^{-k-t-1}
		\left( TU^{t-s} - U^{t-s} T\right) U_0^{s+k}\psi \\
& = Z_{t,s} ((U_0 T) U^{t-s} - (U_0 U^{t-s}) T) \psi,
\end{align*}
where 
\[ Z_{t,s}(A) = \sum_{k=0}^\infty U_0^{-k-t} A U_0^{k+s}. \]
By Lemma \ref{lemZ} below, we know that
\begin{align*} 
|\langle \psi, W_t^*(W_t - W_s)\psi \rangle| 
	& \leq |\langle \psi, Z_{t,s}((U_0 T) U^{t-s}) \psi \rangle| \\
	& \quad +  |\langle \psi, Z_{t,s}(U_0 U^{t-s}) T) \psi \rangle|
		 	\to 0, \quad t,s \to \infty.
\end{align*}
This completes the proof.
\end{proof}
%%Lemma%%%%%%%%%%%%%%%%%%%%%%%%%%%%%%%%%%%%%%%%%%%
\begin{lemma}
\label{lemZ}
{\rm
Let $Y \in \mathscr{T}_1$ and $\{Q(t,s)\}$
be a family of bounded operators with 
$\sup_{t,s}\|Q(t,s)\| < \infty$. 
Then, for all $\psi \in \mathcal{H}_{\rm ac,0}$,
\begin{itemize} 
\item[(1)] $\lim_{t, s \to \infty} \left\langle  \psi, 
	Z_{t,s}(Y Q(t,s) )\psi \right\rangle = 0$;  
\item[(2)] $\lim_{t, s \to \infty} \left\langle  \psi, 
	Z_{t,s}(Q(t,s) Y)\psi \right\rangle = 0$. 
\end{itemize}
}
\end{lemma}
%%%%%%%%%%%%%%%%%%%%%%%%%%%%%%%%%%%%%%%%%%%%%%%
\begin{proof}
Let $Y = \sum_{n=1}^\infty \lambda_n |\psi_n \rangle \langle \phi_n |$
be the canonical expansion of the compact operator $Y$.
Since $Y \in \mathscr{T}_1$, $\sum_n \lambda_n < \infty$. 
Then, by the Cauchy-Schwartz inequality, we have
\begin{align*}
|\left\langle  \psi, 
	Z_{t,s}(Y Q(t,s) )\psi \right\rangle| 
	%= \left|\left\langle  \psi, \sum_{k=0}^{\infty}
		%U_0^{-k-t} Y Q(t,s) U_0^{k+s}\psi \right\rangle \right| \\
& \leq \sum_{n=1}^{\infty} \sum_{k=0}^{\infty} 
	\lambda_n \left| \left\langle U_0^{k+t} \psi, \psi_n \right\rangle
		 \left\langle \phi_n,  Q(t,s) U_0^{k+s}\psi \right\rangle\right| \\
& \leq I_1(t,s)^{1/2} \times I_2(t,s)^{1/2},
\end{align*}
where 
\begin{align*}
& I_1(t) 
	= \sum_{n=1}^{\infty}  \sum_{k=0}^{\infty}
		\lambda_n |\left\langle \psi_n, U_0^{k+t} \psi \right\rangle |^2, \\
& I_2(t,s)
	= \sum_{n=1}^{\infty}  \sum_{k=0}^{\infty}
		\lambda_n  |\left\langle Q(t,s)^*\phi_n, U_0^{k+s}\psi \right\rangle|^2.
\end{align*}
By Lemma \ref{lemac0}, we have
\begin{align} 
%& \label{eql1} I_1(t)  
	%= \sum_{n=1}^\infty \lambda_n 
	%\sum_{k=t}^\infty |\langle\psi_n, U_0^k \psi\rangle|^2
%	\leq  2\pi \sup_\lambda G_\psi(\lambda)^2 \sum_n \lambda_n < \infty, \\
& \notag I_2(t,s)
	%= \sum_{n=1}^{\infty}  \sum_{k=0}^{\infty}
	%	\lambda_n  |\left\langle Q(t,s)^*\phi_n, U_0^{k+s}\psi \right\rangle|^2.
	\leq  2\pi \sup_\lambda G_\psi(\lambda)^2 
		\sup_{t,s}\|Q(t,s)\| \sum_n \lambda_n < \infty,
\end{align}
where we have used the fact that 
%$\psi_n$ and 
$\phi_n$ is a normalized vector.
Let $u_k= \sum_{n=1}^{\infty} 
		\lambda_n |\left\langle \psi_n, U_0^k \psi \right\rangle |^2$. 
Then, similarly to the above, we observe that
$\{u_k\} \in \ell^1(\mathbb{Z})$.
Hence, we have
\[ \lim_{t \to \infty} I_1(t) =  \lim_{t \to \infty}  \sum_{k=t}^\infty u_k = 0. \]
This proves (i).
The same proof works for (ii).	
\end{proof}

%%%Section%%%%%%%%%%%%%%%%%%%%%%%%%%%%%%%%%%%
\section{Asymptotic velocity}
%%%%%%%%%%%%%%%%%%%%%%%%%%%%%%%%%%%%%%%%%%%
\subsection{Proof of Theorem \ref{thm02}}
Let
\[ \mathcal{H}_0 
	= \bigcup_{m=0}^\infty 
		\{ \Psi \in \mathcal{H} \mid \Psi(x) = 0,~ |x| \geq m \} 
	%\quad \mbox{and} \quad 
	%	\mathcal{K}_0 = \mathscr{F} \mathcal{H}_0
	. \]
We use $\mathcal{D}$ to denote 
a subspace of vectors $\Psi \in \mathcal{H}$
whose Fourier transform $\hat \Psi$ are differentiable in $k$ 
with 
\[ \sup_{k \in [0,2\pi)} \left\| \frac{d}{dk} \hat \Psi (k) \right\| < \infty. \]
Note that $\mathcal{H}_0$ is a core for  $\hat x$,
%$\mathcal{H}_0 \subset \mathcal{D}$,
and so is $\mathcal{D}$.
Let $D = \mathscr{F} \hat x \mathscr{F}^{-1}$.
Then, by direct calculation, we know that
$(D \hat \Psi)(k) = i\frac{d}{dk} \hat \Psi(k)$ for $\Psi \in \mathcal{D}$.
We prove the following theorem:
%%Theorem%%%%%%%%%%%%%%%%%%%%%%%%%%%%%%%
\begin{theorem}
\label{thm_resol}
{\rm
Suppose that (A.2) holds.
Then,
\begin{equation} 
\label{Eq_resol}
\slim_{t \to \infty} \left( \frac{\hat x_0(t)}{t} - z \right)^{-1}
	 = (\hat v_0 - z)^{-1},
	\quad z \in \mathbb{C} \setminus \mathbb{R}. 
\end{equation}
 }
\end{theorem}
%%%%%%%%%%%%%%%%%%%%%%%%%%%%%
\begin{proof}
For all $\Psi \in \mathcal{H}$ and $\epsilon > 0$,
there exists a vector $ \Psi_\epsilon \in \mathcal{D}$
such that $\|\Psi - \Psi_\epsilon\| \leq \epsilon$.
Because, by the second resolvent identity,
\begin{align*}
& \left\| \left( \frac{\hat x_0(t)}{t} - z \right)^{-1}\Psi
	 - (\hat v_0 - z)^{-1} \Psi\right\| \\
&\qquad \leq \frac{2\epsilon}{|{\rm Im}z|}
	+ \left\| \left( \frac{\hat x_0(t)}{t} - z \right)^{-1} \Psi_\epsilon
	 - (\hat v_0 - z)^{-1}  \Psi_\epsilon\right\| \\
& \qquad  \leq \frac{2\epsilon}{|{\rm Im}z|}
	+ \frac{1}{|{\rm Im}z|} 
	\left\| \left( \hat v_0 - \frac{\hat x_0(t)}{t} \right) 
	  (\hat v_0 - z)^{-1}  \Psi_\epsilon\right\|, 
\end{align*}
it suffices to prove that
\begin{align*}
\lim_{t \to \infty} 
	\left\| \left( \hat v_0 - \frac{\hat x_0(t)}{t} \right) 
	  (\hat v_0 - z)^{-1}  \Psi\right\| = 0,
	  	\quad  \Psi \in \mathcal{D}.
\end{align*}
Note that
\[ (\hat v_0-z)^{-1} = \mathscr{F}^{-1}
	\left( \int_{[0,2\pi)}^\oplus dk \sum_{j=1,2}
		\left(\frac{ i \lambda_j^\prime(k) }{\lambda_j(k)}  -z\right)^{-1}
			|u_j(k) \rangle \langle u_j(k)|
	\right) \mathscr{F}. \]
Since $\lambda_j(k)$ is analytic and $|\lambda_j(k)|=1$,
we observe from (A.2) that 
$(\hat v_0 - z)^{-1}$ leaves $\mathcal{D}$ invariant.
Hence, we only need to prove that
\begin{align*}
\lim_{t \to \infty} 
	\left\| \left( \hat v_0 - \frac{\hat x_0(t)}{t} \right) 
	\Psi\right\| = 0,
	  	\quad \Psi \in \mathcal{D}.
\end{align*}
By direct calculation, we have
\begin{align*}
& \left\| \left( \hat v_0 - \frac{\hat x_0(t)}{t} \right) \Psi\right\|^2 \\
& \quad = \int_0^{2\pi} dk \left\| 
	\sum_{j=1,2} \left(\frac{i \lambda_j^\prime(k)}{\lambda_j(k)} \right)
		\langle u_j(k), \hat \Psi(k) \rangle u_j(k)
			- \hat U(k)^{-t} \frac{D}{t} \hat U(k)^t \hat \Psi(k)
		\right\|^2 \\
& \quad = \int_0^{2\pi} \frac{dk}{t^2} \left\| 
	\sum_{j=1,2} \lambda_j(k)^t \hat U(k)^{-t}
			\left(i \frac{d}{dk} \langle u_j(k), \hat \Psi(k) \rangle u_j(k) \right)	
		\right\|^2. 
\end{align*}
By the definition of $\mathcal{D}$ and (A.2), we know that
\[ \sup_{k \in [0,2\pi)}
	\left\|\left(i \frac{d}{dk} \langle u_j(k), \hat \Psi(k) \rangle u_j(k) \right)	
		\right\| < \infty. \]
Hence, we have 
\[ \left\| \left( \hat v_0 - \frac{\hat x_0(t)}{t} \right) \Psi\right\|
	= O(t^{-1}), \]
which completes the proof.
\end{proof}

%%%%%%%%%%%%%%%%%%%%%%%%%%%%%%%%%%%%%%%%%%%%%%%
\subsection{Proof of Theorem \ref{thm03}}
%%%%%%%%%%%%%%%%%%%%%%%%%%%%%%%%%%%%%%%%%%%%%%%%%%%
The proof falls naturally into two parts:
%%Theorem%%%%%%%%%%%%%%%%%%%%%%%%%%
\begin{theorem}
\label{thm_p}
{\rm
Let $U$ be a unitary operator on $\mathcal{H}$.
$\hat x(t) = U^{-t} \hat x U^t$ satisfies
\[ \slim_{t \to \infty}  {\rm exp}\left(i \xi \frac{\hat x(t)}{t} \right) \Pi_{\rm p}(U) 
	= \Pi_{\rm p}(U),
	\quad \xi \in \mathbb{R}. \]
}
\end{theorem}
%%Theorem%%%%%%%%%%%%%%%%%%%%%%%%%%
\begin{theorem}
\label{thm_ac}
{\rm
Let $U=SC$ and $U_0 = SC_0$ satisfy (A.1) and (A.2).
Then,
\[ \slim_{t \to \infty}
	{\rm exp}\left( i\xi \frac{\hat x(t)}{t} \right)\Pi_{\rm ac}(U)
		=  {\rm exp}(i \xi \hat v_+)\Pi_{\rm ac}(U),
			\quad \xi \in \mathbb{R}. \]
}
\end{theorem}
%%%%%%%%%%%%%%%%%%%%%%%%%%%%%
%Before proving the above theorems,
%we prove Theorem \ref{thm03}:
\begin{proof}[Proof of Theorem \ref{thm03}]
By (A.3), we have
\begin{align*} 
\slim_{t \to \infty}
	{\rm exp}\left( i\xi \frac{\hat x(t)}{t} \right)
& = \slim_{t \to \infty}
	{\rm exp}\left( i\xi \frac{\hat x(t)}{t} \right)
		(\Pi_{\rm p}(U) + \Pi_{\rm p}(U)) \\
& = \Pi_{\rm p}(U) +  {\rm exp}(i \xi \hat v_+)\Pi_{\rm ac}(U).			
\end{align*}
This prove the theorem.
\end{proof}
%%%%%%%%%%%%%%%%%%%%%%%%%%%%%%%%%%%%%%%%%%
It remains to prove Theorems \ref{thm_p} and \ref{thm_ac}.
%%%%%%%%%%%%%%%%%%%%%%%%%%%%%%%%%%%%%%%%%
\begin{proof}[Proof of Theorem \ref{thm_p}]
Let $\mathcal{H}_{\rm p}(U)$ be the direct sum of 
all eigenspaces of $U$. 
It suffices to prove that, for $\Psi \in \mathcal{H}_{\rm p}(U)$,
\[ \slim_{t \to \infty}  {\rm exp}\left(i \xi \frac{\hat x(t)}{t} \right) \Psi
	=\Psi . \]
Let $\lambda_n$ be the eigenvalues of $U$ and
take an ONB $\{\eta_n\}_{n=1}^\infty$ of $\mathcal{H}_{\rm p}$
such that $U\eta_n = \lambda_n \eta_n$.
We have $\Pi_{\rm p}(U) = \sum_n |\eta_n \rangle \langle \eta_n|$.
Let $\epsilon >0$. For sufficiently large $N$,
$\Psi_N = \sum_{n=1}^N \langle \eta_n, \Psi \rangle \eta_n$
satisfies $\|\Psi - \Psi_N\| \leq \epsilon$. 
Then,
\begin{align*}
\left\|{\rm exp}\left( i \xi \frac{\hat x(t)}{t} \right)\Psi - \Psi \right\|
\leq 2 \epsilon 
	+ \left\|{\rm exp}\left( i \xi \frac{\hat x(t)}{t} \right)\Psi_N - \Psi_N \right\|.
\end{align*}
By direct calculation, we have
\begin{align}
& \notag
\left\|{\rm exp}\left( i \xi \frac{\hat x(t)}{t} \right)\Psi_N - \Psi_N \right\| 
 = \left\| \left( {\rm exp}\left( i \xi \frac{\hat x}{t} \right) -1 \right) 
	U^t \Psi_N\right\| \\
& \notag
\qquad = \left\| \sum_{n=1}^N \lambda_n^t \langle \eta_n, \Psi \rangle
	\left( {\rm exp}\left( i \xi \frac{\hat x}{t} \right) -1 \right) 
	  \eta_n \right\| \\
& \label{eq23} 
\qquad \leq \sum_{n=1}^N |\langle \eta_n, \Psi \rangle|
	 \left\| \left( {\rm exp}\left( i \xi \frac{\hat x}{t} \right) -1\right) 
	  \eta_n \right\|.
\end{align}
Since $\lim_{t \to \infty} |1-e^{i \xi x/t}|=0$,
$|1-e^{i \xi x/t}| \leq 2$ and 
$\sum_{x} \|\eta_n(x)\|_{\mathbb{C}^2}^2 = \|\eta_n\|^2 < \infty$,
we have
\begin{align*} 
\lim_{t \to \infty} \left\| \left( {\rm exp}\left( i \xi \frac{\hat x}{t} \right) -1\right) 
	  \eta_n \right\|^2
= \lim_{t \to \infty} \sum_{x \in \mathbb{Z}} |e^{i \xi x/t}-1|^2 \|\eta_n(x)\|_{\mathbb{C}^2}^2 = 0,
\end{align*}
which, combined with \eqref{eq23}, completes the proof.
\end{proof}
%%%%%%%%%%%%%%%%%%%%%%%%%%%%%%%%%%%%%%%%%%%
\begin{lemma}
\label{lem_com}
{\rm
$[U_0, {\rm exp}(i \xi \hat v_0)] = 0$.
}
\end{lemma}
%%%%%%%%%%%%%%%%%%%%%%%%%%%%
\begin{proof}
By direct calculation, we have
\begin{align*}
[ U_0, {\rm exp}(i \xi \hat v_0)]
	& = \slim_{t \to \infty} 
		\left[ U_0, {\rm exp}\left(i \xi \frac{\hat x_0(t)}{t}\right)\right] \\
	& = \slim_{t \to \infty} U_0 
		\left\{ {\rm exp}\left(i \xi \frac{\hat x_0(t)}{t}\right)
					- {\rm exp}\left(i \xi \frac{\hat x_0(t+1)}{t}\right) \right\} = 0. 
\end{align*}
\end{proof}
%%%%%%%%%%%%%%%%%%%%%%%%%%%%%%%%%%%%%%%%%%%%%%
\begin{proof}[Proof of Theorem \ref{thm_ac}]
By (A.1) and (A.2), Theorems \ref{thm01} and \ref{thm02} hold.
Then, $W_+$ is a unitary operator from $\mathcal{H}_{\rm ac}(U_0)$
to $\mathcal{H}_{\rm ac}(U)$.
Hence, we have
\[ {\rm exp}(i \xi \hat v_+) \Pi_{\rm ac}(U)
	= W_+ {\rm exp}(i \xi \hat v_0) W_+^*\Pi_{\rm ac}(U). \]
By direct calculation, we observe that
\begin{align*}
I(t) &:=  {\rm exp}\left( i\xi \frac{\hat x(t)}{t} \right)\Pi_{\rm ac}(U)
		 -  {\rm exp}(i \xi \hat v_+)\Pi_{\rm ac}(U) \\
& = W_t {\rm exp}\left( i\xi \frac{\hat x_0(t)}{t} \right) 
		W_t^* \Pi_{\rm ac}(U)
		 - W_+ {\rm exp}(i \xi \hat v_0) W_+^*\Pi_{\rm ac}(U) \\
& =: \sum_{j=1}^3 I_j(t),  
\end{align*}
where
\begin{align*}
& I_1(t) = W_t {\rm exp}\left( i\xi \frac{\hat x_0(t)}{t} \right) 
		\left(W_t^* - W_+^* \right)\Pi_{\rm ac}(U), \\
& I_2(t) = W_t\left( {\rm exp}\left( i\xi \frac{\hat x_0(t)}{t} \right) 
		-  {\rm exp}(i \xi \hat v_0) \right)
		W_+^*  \Pi_{\rm ac}(U), \\
& I_3(t) = \left( W_t - W_+ \right) 
		 {\rm exp}(i \xi \hat v_0) 
		W_+^*  \Pi_{\rm ac}(U).
\end{align*}
Because $W_t$ and ${\rm exp}(i \xi \hat x_0(t)/t)$ are uniformly bounded,
we know from Theorems \ref{thm01} and \ref{thm02}
that $\slim_{t \to \infty}I_1(t) = \slim_{t \to \infty}I_2(t) = 0$.
Hence, we have
\begin{align*} 
I(t) & = \left( W_t - W_+ \right) 
		 {\rm exp}(i \xi \hat v_0) 
		W_+^*  \Pi_{\rm ac}(U) + o(1) \\
	& =  \left( W_t - W_+ \right) 
		\Pi_{\rm ac}(U_0) {\rm exp}(i \xi \hat v_0) 
			W_+^*  \Pi_{\rm ac}(U) \\
	& \quad + \left( W_t - W_+ \right) 
			[ {\rm exp}(i \xi \hat v_0), \Pi_{\rm ac}(U_0)]
			W_+^*  \Pi_{\rm ac}(U) + o(1),
\end{align*}
where we have used the fact that ${\rm Ran} W_+^* 
=\mathcal{H}_{\rm ac}(U_0)$.
Since, by Lemma \ref{lem_com},
$[ {\rm exp}(i \xi \hat v_0), \Pi_{\rm ac}(U_0)]=0$,
we obtain from Theorem \ref{thm01}, that
$\slim_{t \to \infty} I(t)=0$.
This completes the proof.
\end{proof}

%%%%%ACKNOWLEDGMENT%%%%%%%%%%%%%%%%%%%%%%%%%
{\bf Acknowledgement}
This work was supported by Grant-in-Aid for Young Scientists (B) (No. 26800054).
%%%%%%%%%%%%%%%%%%%%%%%%%%%%%%%%%%%%%%%%%%%%%%

\end{document}